\documentclass[a4paper,10pt]{article}
\usepackage[utf8x]{inputenc}
\usepackage{amsmath,amsthm,amssymb}
\usepackage{amsfonts}
\usepackage{amssymb}

\title{Separation of analytic sets by rectangles of low complexity}
\author{Rafael Zamora\footnote{Universit\'e Paris 6, Institut de Math\'ematiques de Jussieu, Projet Analyse Fonctionnelle, Couloir 15-16, 5\`eme  \'etage, Case 247, 4, place Jussieu, 75 252 Paris Cedex 05, France.
rafael.zamora@imj-prg.fr} \footnote{I would like to thank Dominique Lecomte for all the remarks and the extensive proofreading he provided, his help was essential.}}
\date{ }

\newtheorem{thm}{Theorem}[section]
\newtheorem{lem}[thm]{Lemma}
\newtheorem{Cor}[thm]{Corollary}
\newtheorem{Pro}[thm]{Proposition}
\newtheorem{con}[thm]{Conjecture}

\theoremstyle{definition}
\newtheorem*{defi}{Definition}
\newtheorem*{nota}{Notation}

\begin{document}

\maketitle

\newcommand{\cantor}{2^\omega}
\newcommand{\baire}{\omega^\omega}
\newcommand{\bbA}{\mathbb{A}}
\newcommand{\bbB}{\mathbb{B}}
\newcommand{\bbX}{\mathbb{X}}
\newcommand{\bbY}{\mathbb{Y}}
\newcommand{\bbW}{\mathbb{W}}
\newcommand{\bbP}{\mathbb{P}}
\newcommand{\open}{\mathbf\Sigma^0_1}
\newcommand{\close}{\mathbf{\Pi}^0_1}
\newcommand{\gdelta}{\mathbf{\Pi}^0_2}
\newcommand{\fsigma}{\mathbf\Sigma^0_2}
\newcommand{\taudos}{{\tau_2}}
\newcommand{\semirecursive}{{\Sigma^0_1}}
\newcommand{\Acero}{\Pi_0[A]}
\newcommand{\Auno}{\Pi_1[A]}
\newcommand{\rationals}{\mathbb{P}_f}
\newcommand{\irrationals}{\mathbb{P}_\infty}

\begin{abstract}
We provide Hurewicz tests for the separation of disjoint analytic sets by rectangles of the form $\Gamma\times\Gamma'$ for $\Gamma,\Gamma'\in\{\open,\close,\gdelta\}$.
\end{abstract}

\section{Introduction} 

One of the turning points in Descriptive Set Theory was the realization of the fact that the continuous images of Borel sets are not necessarily Borel, but that they define a new class, that of the analytic sets. However, this class kept some of the nice structural properties of other well known classes. 

A remarkable example was Lusin's Separation Theorem (see for example \cite{K95} for this theorem, basic theory and notation), which states that for any two disjoint analytic subsets of a Polish space, one can find a Borel set which contains one and does not intersect the other (i.e., separates the first one from the other). A nice refinement was offered by A. Louveau and J. Saint-Raymond in \cite{LSR87}, where they gave a test to recognize when two disjoint analytic sets can be separated by a set of a given Borel class. 

A different, but somewhat related question was first answered by A. Kechris, S. Solecki and S. Todorcevic in \cite{KST99}: given an analytic graph $G$ (i.e., a symmetric irreflexive relation) on $X$, when does this graph have a countable Borel coloring, i.e., a Borel function from $X$ to $\omega$ whose inverse images of points are $G$-discrete. This question also makes sense for directed graphs, i.e., irreflexive relations. 

We can viewed this as a separation question on $X\times X$: we are trying to separate the diagonal from $G$ by an union of countably many disjoint Borel ``squares'' (i.e., sets of the form $C\times C$). The more general problem of separation of analytic sets by a countable union of Borel rectangles (i.e, sets of the form $C\times D$) was in some way treated by J.H. Silver (see for example, \cite{S09}). Similar questions were then answered by L.A Harrington, A. Kechris and A. Louveau in \cite{HKL90} and by D. Lecomte in \cite{L07}. 

All these results have something in common. They are what is commonly call a Hurewicz-like test. This consist of an example which is usually simple to understand and which in someway embeds into the sets which do not satisfy a certain property. So for example, in Hurewicz's original theorem, we obtain a set that continuously embeds in all analytic sets which are not $\fsigma$.

In \cite{L07} and \cite{L09}, D. Lecomte studies the separation of analytic sets by a countable union of Borel rectangles. In order to do this, he introduces the following quasi-order. He finds minimal examples for this quasi-order for pairs of sets without this property.

\begin{defi} For $e\in 2$, let $X_e,Y_e$ be Polish spaces, and $A_e,B_e\subseteq X_e\times Y_e$. We say that $(X_0,Y_0,A_0,B_0)$ reduces to $(X_1,Y_1,A_1,B_1)$ if there are continuous functions $f:X_0\to X_1$ and $g:Y_0 \to Y_1$ such that
\[A_0\subseteq (f\times g)^{-1}(A_1)\]
and
\[B_0\subseteq (f\times g)^{-1}(B_1).\]
In this case, we write $(X_0,Y_0,A_0,B_0)\leq(X_1,Y_1,A_1,B_1)$.
\end{defi}

Later, D. Lecomte and M. Zeleny studied in \cite{LZ12} how to solve a different question using this same quasi-order. The problem is to characterize when an analytic set is separable from another by a countable union of sets of fixed Borel complexity. They also studied the problem of characterizing when an analytic digraph has a coloring of bounded Borel complexity. In particular, they proved the following conjecture (also proposed in \cite{LZ12}) when $\xi=1,2$.

\begin{con}\label{LZ2} (Lecomte, Zeleny) For each $0<\xi<\omega_1$ there are Polish spaces $\mathbb{X}_\xi,\mathbb{Y}_\xi$ and analytic subsets $\mathbb{A}_\xi,\mathbb{B}_\xi$ of $\mathbb{X}_\xi\times\mathbb{Y}_\xi$ such that for all Polish spaces $X,Y$ and analytic disjoint $A,B\subseteq X\times Y$, exactly one of the following holds:
\begin{enumerate}
	\item $B$ is separable from $A$ by a $(\mathbf\Sigma^0_\xi\times\mathbf\Sigma^0_{\xi})_\sigma$,
	\item $(\mathbb{X}_\xi,\mathbb{Y}_\xi, \mathbb{B}_\xi,\mathbb{A}_\xi)\leq (X,Y,B,A)$.
\end{enumerate}
\end{con}

We clarify some notation: given a class of sets $\Gamma$, as usual $\Gamma(X)=\{A\in\Gamma| A\subseteq X\}$. Also $\Gamma\times\Gamma'=\{A\times B| A\in \Gamma, B\in\Gamma'\}$, and finally $\Gamma_\sigma=\{\cup_{n\in\omega} A_n|A_n\in\Gamma\}$. We also use $\Pi_0$ and $\Pi_1$ for the projection on the first and second coordinate respectively.

How can we characterize the separability by a Borel rectangle? This is probably folklore, but we will provide a proof later on.

\begin{Pro} \label{Borel} Let $X,Y$ be Polish spaces, and let $A,B$ be $\mathbf\Sigma^1_1$ subsets of $X\times Y$. The following are equivalent:
\begin{enumerate}
	\item $A$ is not separable from $B$ by a $\mathbf\Delta^1_1\times\mathbf\Delta^1_1$ set, 
	\item $(\Pi_0[A]\times\Pi_1[A])\cap B\neq\emptyset$, so in particular $A$ is not separable from $B$ by an arbitrary rectangle,
	\item $\big(2,2,\{(0,0),(1,1)\},\{(0,1)\}\big)\leq (X,Y,A,B)$.
\end{enumerate}
\end{Pro}

One would like to refine this result as the Louveau-Saint Raymond Theorem is a refinement of the Lusin Theorem. We want a test characterizing the separabilty of analytic sets by a rectangle with sides of bounded Borel complexity. In Section \ref{sec1} we provide some first results in this direction. In particular, we  characterize when a Borel rectangle is in fact a $\mathbf\Sigma^0_\xi$ rectangle. For the reminder, we  use effective descriptive set theory. In Section \ref{sec2} we use some effective topologies to characterize separation of $\Sigma^1_1$ sets by $\Gamma\times\Gamma'$ sets for several  pairs $(\Gamma, \Gamma')$. 

The rest of the article will be dedicated to the proof of the main theorem, which will be done individually for each pair. 

\begin{thm}\label{Main} Let $\Gamma,\Gamma'\in\{\open,\close,\gdelta\}$ or $\Gamma\times\Gamma'=\close\times\fsigma$. There is a finite antichain $\mathcal{C}$ of the class of quadruples of the form $(Z,W,C,D)$, where $Z,W$ are Polish spaces, and $C,D\subseteq Z\times W$ are disjoint analytic subsets, such that, for any Polish spaces $X,Y$ and disjoint analytic subsets $A,B$ of $X,Y$, exactly one of the following holds:
\begin{enumerate}
	\item $A$ is separable from $B$ by a $\Gamma\times\Gamma'$ set,
	\item there is $(\bbX,\bbY,\bbA,\bbB)\in\mathcal{C}$ such that $(\bbX,\bbY,\bbA,\bbB)\leq(X,Y,A,B)$.
\end{enumerate}
\end{thm}

This antichain $\mathcal{C}$ is usually referred to as an antichain basis. In general, we will describe a general process to obtain this antichain.

In Section \ref{sec3}, we will obtain $\mathcal{C}$ for open rectangles. We note that usually, the antichain for separation of two disjoint analytic sets does not differ a lot from the special case when these two analytic are complements of each other. This is not the case here, the antichain not only being quite different, but having two elements instead of one. 

In the rest of the sections, we will offer $\mathcal{C}$ for the rest of the combinations for $\Gamma,\Gamma'$. 

\section{First results}\label{sec1}

The results here are all consequences of classical results.

\begin{proof}[Proof of Proposition \ref{Borel}]  $(3\Rightarrow 1).$ Suppose that $A$ is separable from $B$ by $C\times D$. Then, $f^{-1}(C)\times g^{-1}(D)$ must also separate $\{(0,0),(1,1)\}$ from \{(0,1)\}. This is clearly absurd.

$(1\Rightarrow 2).$ We argue by contradiction.  In this case,  $\Pi_0[A]$ is disjoint from $\{x\in X| \exists y\in\Pi_1[A] \big((x,y)\in B\big)\}$. Since these sets are analytic, Lusin's Theorem gives a Borel set $C$ which separates $\Pi_0[A]$ from the other set. 

Note that $\big(C\times\Pi_1[A]\big)\cap B=\emptyset$. Therefore,  $\Pi_1[A]$ is disjoint from $\{y\in Y| \exists x\in C \big((x,y)\in B\big)\}$. Again, Lusin's theorem gives a Borel set $D$ which separates them. Thus, $C\times D$ separates $A$ from $B$.

$(2\Rightarrow 3).$ Let $(x,y)\in(\Pi_0[A]\times\Pi_1[A])\cap B$. There is $y'\in Y$ (respectively $x'\in X$) which witness the fact that $x\in \Pi_0[A]$ (resp. $y\in \Pi_1[A]$). Define
\[\begin{array}{rc}
	f(e)&=\left\{\begin{array}{lc}
							x & \mbox{if } e=0,\\
							x'& \mbox{if } e=1,
							\end{array}	
				\right.\\
	g(e)&=\left\{\begin{array}{lc}
							y & \mbox{if } e=1,\\
							y'& \mbox{if } e=0.
							\end{array}	
				\right.
\end{array}\] 

Finally, $f\times g$ is clearly  the required reduction.
\end{proof}

Note that Lusin's Theorem is also a consequence of Proposition \ref{Borel}. Indeed, if two analytic sets $C,D$ are disjoint, then $C\times D$ is disjoint from the diagonal. Proposition \ref{Borel} gives a Borel set separating $C$ from $D$.

Now, we would like to answer the following question: can we find a Hurewicz-like test to decide if a particular Borel set in the plane is a rectangle of a particular complexity? As we will see, this is in fact just an easy application of the previously mentioned result by A. Louveau and J. Saint Raymond in \cite{LSR87}.  For each $0<\xi<\omega_1$, let $S_\xi\in \mathbf\Pi^0_\xi(\cantor)\backslash\mathbf\Sigma^0_\xi(\cantor)$. Define

\[\bbA_\xi=\{(0\alpha,1^\infty)|\alpha\in S_\xi\}\cup\{(1^\infty,0\alpha)|\alpha\in S_\xi\},\]
\[\bbB_\xi=\{(0\alpha,0\alpha)|\alpha\notin S_\xi\}.\]

\begin{Pro}\label{ProBR} Let $0<\xi<\omega_1$, $X,Y$ be Polish spaces, and $A$ be a Borel subset of $X\times Y$. Exactly one of the following must hold:
\begin{enumerate}
	\item $A$ is a $\mathbf\Sigma^0_\xi$ rectangle,
	\item $(\cantor,\cantor,\bbA_\xi,\bbB_\xi)\leq(X,Y,A,\neg A).$
\end{enumerate}
\end{Pro}
\begin{proof}
 The exactly part comes from the fact that $\bbA_\xi$ cannot be separated from $\bbB_\xi$ by a $\mathbf\Sigma^0_\xi$ rectangle; which in turn is because $S_\xi$ is not separable from $\neg S_\xi$ by a $\mathbf\Sigma^0_\xi$ set. 

So suppose that $A$ is not a $\mathbf\Sigma^0_\xi$ rectangle. There are two cases.

First case: $A$ is not a rectangle. In this particular case, $(\Pi_0[A]\times\Pi_1[A])\cap \neg A\neq \emptyset$. So take $(x,y)$ in this intersection. Pick $x'\in X, y'\in Y'$ such that $(x,y'),(x',y)\in A$. The following functions define the required reduction $f\times g$:

\[\begin{array}{rc}
	f(\alpha)&=\left\{\begin{array}{lc}
							x & \mbox{if } \alpha(0)=0,\\
							x'& \mbox{if } \alpha(0)=1,
							\end{array}	
				\right.\\
	g(\alpha)&=\left\{\begin{array}{lc}
							y & \mbox{if } \alpha(0)=0,\\
							y'& \mbox{if } \alpha(1)=1.
							\end{array}	
				\right.
\end{array}\]

Second case: $A=\Acero\times\Auno$. Then, at least one of the sides must not be a $\mathbf\Sigma^0_\xi$ set. Without loss of generality, suppose that it is $\Acero$. This gives a continuous function $\hat{f}:\cantor\to X$ such that
\[S_\xi\subseteq\hat{f}^{-1}(\Acero)\]
\[\neg S_\xi \subseteq\hat{f}^{-1}(\neg\Acero)\]

Choose $(x,y)\in A$. Note that, $\big(\hat{f}(\alpha),y\big)\in A$ for all $\alpha\in S$.
Define $f:\cantor\to X$ and $g:\cantor\to Y$ by:

\[f(e\alpha)=\left\{\begin{array}{lc}
							\hat{f}(\alpha) & \mbox{if } e=0,\\
							x & \mbox{if } e=1,
							\end{array}	
				\right.\]
\[g(\alpha)=y. \]

So in both cases we obtain a reduction.
\end{proof}

\section{Topological characterizations}\label{sec2}

We will use effective methods to solve our problems of classical type. We remember that given a recursively presented Polish space $X$ (see \cite{M1980} for basic definitions), there are certain topologies on $X$ whose closure operation captures the separability of $\Sigma^1_1$ subsets by subsets in a particular class of the Borel hierarchy.  These topologies, as well as the following statement, were introduced by Louveau in \cite{LO1980}. In the sequel, everything can be relativized to an element of $\baire$. Let $\tau_1(X)$ denote the original topology on $X$. If $1<\xi<\omega_1$, then $\tau_{\xi}(X)$ is the topology generated by all $\Sigma^1_1\cap\mathbf\Pi^0_{<\xi}$ subsets of $X$, where $\mathbf\Pi^0_{<\xi}:=\cup_{\nu<\xi}\mathbf\Pi^0_\nu$.  We can then state our previous remark formally.

\begin{thm}\label{LO1} (Louveau) Let $0<\xi<\omega^{CK}_1$, $X$ be a recursively presented space, and $A,B$ be disjoint $\Sigma^1_1$ subsets of $X$. The following are equivalent:
\begin{enumerate}
	\item $A$ is separable from $B$  by a $\mathbf{\Sigma}^0_\xi$ subset,
	\item $A\cap \overline{B}^{\tau_\xi}=\emptyset$.
\end{enumerate} 
\end{thm}

We also note that as proved in \cite{LO1980}, if $B$ is $\Sigma^1_1$, then $\overline{B}^{\tau_\xi}$ is also $\Sigma^1_1$.

This was first extended to the case of dimension two by D. Lecomte using a similar set of topologies. Then D. Lecomte and M. Zeleny found in \cite{LZ12} a similar statement for countable unions of $\mathbf{\Sigma}^0_\xi$ rectangles. We give a further generalized statement, which follows with the exact same proof as the original one.

\begin{lem}\label{LZ1} (Lecomte, Zeleny) Let $0<\xi,\xi'<\omega^{CK}_1$, $X,Y$ be recursively presented Polish spaces, and $A$,$B$ be disjoint $\Sigma^1_1$ subsets of $X\times Y$. The following are equivalent:
\begin{enumerate}
	\item $A$ is separable from $B$  by a $(\mathbf\Sigma^0_\xi\times \mathbf\Sigma^0_{\xi'})_\sigma$ set,
	\item $A\cap \overline{B}^{\tau_\xi\times\tau_{\xi'}}=\emptyset$.
\end{enumerate}  
\end{lem}

We can deduce from this a dual condition for the separability by one $\mathbf\Pi^0_\xi$ rectangle.

\begin{lem}\label{lem1} Let $0<\xi,\xi'<\omega^{CK}_1$,  $X,Y$ be recursively presented Polish spaces, and $A$, $B$ be disjoint $\Sigma^1_1$ subsets of $X\times Y$. The following are equivalent:
	\begin{enumerate}
		\item $A$ is separable from $B$ by a $(\mathbf\Pi^0_\xi\cap \Sigma^1_1)\times (\mathbf\Pi^0_{\xi'}\cap \Sigma^1_1)$ set,
		\item $A$ is separable from $B$ by a $(\mathbf\Pi^0_\xi\cap \Delta^1_1)\times (\mathbf\Pi^0_{\xi'}\cap \Delta^1_1)$ set,
		\item $A$ is separable from $B$ by a $\mathbf\Pi^0_\xi\times \mathbf\Pi^0_{\xi'}$ set,
		\item $B\cap \overline{\Pi_0[A]\times\Pi_1[A]}^{\tau_\xi\times\tau_{\xi'}}=\emptyset$.
	\end{enumerate}  
\end{lem} 
\begin{proof}
$(1\Rightarrow 2)$. Let $C\times D\in(\mathbf\Pi^0_\xi\cap \Sigma^1_1)\times (\mathbf\Pi^0_{\xi'}\cap \Sigma^1_1)$ such that it separates $A
$ from $B$. Then, $C$ is a $\mathbf\Pi^0_\xi$ which separates $\Pi_0[A]$ from $\Pi_0[(X\times D)\cap B]$. Since these two sets are $\Sigma^1_1$, by \cite[Theorem B]{LO1980} there is $C'\in\mathbf\Pi^0_\xi\cap\Delta^1_1$ which separates $\Acero$ from $\Pi_0[(X\times D)\cap B]$. From this, we obtain that $\Auno$ is separated from $\Pi_1[(C'\times Y)\cap B]$ by $D$. Since $D$ is $\mathbf\Pi^0_{\xi'}$, by \cite[Theorem B]{LO1980} we obtain $D'\in\mathbf\Pi^0_\xi\cap\Delta^1_1$, which separates these two sets. In particular, $C'\times D'$ separates $A$ from $B$.
 
$(2\Rightarrow 3).$ It is obvious.

$(3\Rightarrow 4).$ Note that if $A$ is separable from $B$ by a $\mathbf\Pi^0_\xi\times\mathbf\Pi^0_{\xi'}$ set, then so is $\Pi_0[A]\times\Pi_1[A]$. It follows that $B$ will be separable from $\Pi_0[A]\times\Pi_1[A]$ by a $(\mathbf\Sigma^0_\xi\times\mathbf\Sigma^0_{\xi'})_\sigma$ set, namely the complement of the $\mathbf\Pi^0_\xi\times\mathbf\Pi^0_{\xi'}$ set. Finally, apply Lemma \ref{LZ1}.

$(4\Rightarrow 1).$ It is the fact that $\overline{C}^{\tau_\xi}$ is a $\mathbf\Pi^0_\xi\cap \Sigma^1_1$, if $C$ is $\Sigma^1_1$, as shown in \cite{LO1980}.
\end{proof}

We would like to find a similar topological characterization for the separability by a set in some of the other classes $\Gamma\times\Gamma'$. We start with a sufficient condition for the separability by a $\mathbf\Sigma^0_\xi\times\mathbf\Sigma^0_{\xi'}$ set.

\begin{lem} \label{L2} Let $0<\xi,\xi'<\omega^{CK}_1$, $X,Y$ be recursively presented Polish spaces, and $A,B$ be disjoint subsets of $X\times Y$. If $A$ is not separable from $B$ by a $\mathbf\Sigma^0_\xi\times\mathbf\Sigma^0_{\xi'}$ set, then at least one of the following holds.
\begin{enumerate}
	\item There is $x\in\Pi_0[A]$ such that, for every $\tau_\xi$-open neighborhood $U$ of $x$, there is  $y\in \Pi_1[A]$ such that, for every $\tau_{\xi'}$-open neighborhood $V$ of $y$, $(U\times V)\cap B\neq\emptyset$.
	\item There is $y\in \Pi_1[A]$ such that, for every $\tau_{\xi'}$-open neighborhood $V$ of $y$, there is $x\in\Pi_0[A]$ such that, for every $\tau_\xi$-open neighborhood $U$ of $x$, $(U\times V)\cap B\neq\emptyset$.
\end{enumerate}
\end{lem}
\begin{proof} Suppose that neither 1. nor 2. holds. Since $\tau_\xi(X)$ and $\tau_{\xi'}(Y)$ have countable basis $\{U_n\}_{n\in\omega}$ and $\{V_m\}_{m\in\omega}$ respectively, for each $x\in \Pi_0[A]$ we can find a $n_x$ that witnesses the negation of 1. We can find respectively a $m_y$ that witnesses the negation of 2. for each $y\in\Pi_1[A]$. Let $N:=\{n_x|x\in\Pi_0[A]\}$ and $M:=\{m_y|y\in\Pi_1[A]\}$. 

Now, each $n\in N$ satisfies the following. For each $y\in\Pi_1[A]$ there exists an $m'_{n,y}\in\omega$ such that, for all $n'\leq m_{y}$ in $N$, $(U_{n'}\times V_{m'_{n,y}})\cap B=\emptyset$ and $y\in V_{m'_{n,y}}\subseteq V_{m_y}$. Indeed, each $n'\in N$ is a $n_x$ for some $x$, and so, we only need to take $V_{m'_{n,y}}$ as a basic open which witnesses the negation of 1 for each $n_x$. Likewise, we find for each $m\in M$ and each $x\in \Pi_0[A]$  a $n'_{m,y}$ which satisfies the dual statement. Set $N'=\{n'_{m,x}|m\in M, x\in \Pi_0[A]\}$ and $M':=\{m'_{n,y}|n\in N, y\in \Pi_1[A]\}$.

Finally, we claim that $(\cup_{n\in N'}U_n)\times(\cup_{m\in M'}V_{m})$
separates $A$ from $B$. It obviously contains $A$. To see that it does not intersect $B$ suppose it does, so that there are $n'\in N'$ and $m'\in M'$, such that $(U_{n'}\times V_{m'})\cap B\neq\emptyset$.  We must then have $n'=n'_{m,x}$ for some $m,x$ and $m'=m'_{n,y}$ for some $n,y$. In particular, we can suppose that $n_x\leq m_y$ (the other case is similar). Then, $U_{n'}\times V_{m'}\subseteq U_{n_x}\times V_{m'}$, and by the definition of $m'_{n,y}$ this last one should not intersect $B$.  
\end{proof}

We would like to point out that in Lemma \ref{L2}, there are no hypothesis on the complexity of $A$ nor $B$. Also, for the open case, the effective hypothesis on the spaces is useless, the theorem still holds with general Polish spaces. In fact, we can obtain a different version in the open case, if we add some hypothesis on the complexity of $A$ and $B$. 

\begin{lem} Let $X,Y$ be recursively presented Polish spaces, and $A,B$ be disjoint $\Sigma^1_1$ subsets of $X\times Y$. The following are equivalent:
\begin{enumerate}
	\item $A$ is not separable from $B$ by a $\open\times\open$ set,
	\item $A$ is not separable from $B$ by a $(\open\cap \Delta^1_1)\times(\open\cap \Delta^1_1)$ set,
	\item  At least one of the following holds,
\begin{enumerate}
	\item there is $x\in\Pi_0[A]$ such that, for every $\tau_1$-open neighborhood $U$ of $x$, there is  $y\in \Pi_1[A]$ such that, for every $\tau_{1}$-open neighborhood $V$ of $y$, $(U\times V)\cap B\neq\emptyset$,
	\item there is $y\in \Pi_1[A]$ such that, for every $\tau_{1}$-open neighborhood $V$ of $y$, there is $x\in\Pi_0[A]$ such that, for every $\tau_1$-open neighborhood $U$ of $x$, $(U\times V)\cap B\neq\emptyset$.
\end{enumerate}
\end{enumerate}
\end{lem}
\begin{proof} $(1.\Rightarrow 2.)$. It is obvious.

$(2.\Rightarrow 3.)$. As in the previous Lemma, we will show the contrapositive statement. The proof is basically the same, but one needs to be sure that the choices can be made effectively. For this, we fix basis  $(U_n)$ and $(V_m)$ for $X$ and $Y$ respectively, made of $\semirecursive$ sets. 

Note then that, for all $x\in\Acero$, there is $n\in\omega$ such that 
$x\in U_n \land \big(\forall y\in \Auno\exists m\in\omega \big((y\in V_m)\land (U_n\times V_m)\cap B\big)=\emptyset\big)$. 
Note this statement is $\Pi^1_1$ in $x$ and $n$, so by the $\Delta^1_1$-selection principle, we can make the choice of $n_x$ in a $\Delta^1_1$ way. By a similar argument, we can choose $m_y$ in a $\Delta^1_1$ way.  This implies that $N:=\{n_x| x\in \Acero\}$ and $M:=\{m_y|y\in \Auno\}$ are $\Sigma^1_1$. 

Similarly, if we fix $n\in N$ and $y\in\Auno$, there is an $m\in\omega$ such that $(y\in V_{m})\land(V_{m}\subseteq V_{m_y})\land\forall n'\leq m_y \big((U_{n'}\times V_{m})\cap B\big)=\emptyset)$. Again, note that this statement is $\Pi^1_1$ in $n,m$ and $y$, so we can make the choice of $m'_{n,y}$ in a $\Delta^1_1$ way. A similar argument allows us to choose $n'_{m,y}$. This and the previous paragraph show that $N':=\{n'_{m,x}|m\in M , x\in\Acero\}$ and $M':=\{m'_{n,y}|n\in N , y\in\Auno\}$ are $\Sigma^1_1$. 

So with the same argument as in the proof of Lemma \ref{L2}, $(\cup_{n\in N'}U_n)\times (\cup_{m\in M'}V_m)$ is a $(\open\cap\Sigma^1_1)\times(\open\cap\Sigma^1_1)$ set that separates $A$ from $B$. We claim that this is enough. In fact, following the proof of Lemma \ref{lem1}, by applying \cite[Theorem B]{LO1980} twice we can obtain a $(\open\cap\Delta^1_1)\times(\open\cap\Delta^1_1)$ set separating $A$ from $B$.

$(3.\Rightarrow 1.)$. Again, we prove the contrapositive statement. If $A$ is separated from $B$ by a $\open\times\open$ set, then each of the sides of this rectangle will witness the negation of 1. and of 2. respectively. 
\end{proof}

We note that $(1.\Rightarrow 2.)$ and $(2.\Rightarrow 3)$ can be generalized with a similar proof to the $\mathbf\Sigma^0_\xi\times\mathbf\Sigma^0_{\xi'}$ cases. 

We would like to find a general condition for the separabilty by a $\mathbf\Pi^0_\xi\times\mathbf\Sigma^0_{\xi'}$ set. For now, we consider the cases where either $\xi$ or $\xi'$ is equal to 1.

\begin{lem} \label{L3} Let $0<\xi<\omega^{CK}_1$, $X,Y$ be recursively presented Polish spaces, and $A,B$ be disjoint $\Sigma^1_1$ subsets of $X\times Y$. 
\begin{enumerate}
	\item\label{l1} The following are equivalent:
	\begin{enumerate}
		\item\label{l1.1} $A$ is separable from $B$ by a $\mathbf\Pi^0_1\times\mathbf\Sigma^0_\xi$ set,
		\item\label{l1.3} $A$ is separable from $B$ by a $(\mathbf\Pi^0_1\cap\Delta^1_1)\times(\mathbf\Sigma^0_\xi\cap\Delta^1_1)$ set,
		\item\label{l1.2} for all $y\in \Pi_1[A]$ there is a $\tau_\xi$-open neighborhood $V$ of $y$ such that
		\[(\overline{\Pi_0[A]}\times V)\cap B=\emptyset.\]
	\end{enumerate}
	\item\label{l.2} The following are equivalent:
	\begin{enumerate}
		\item\label{l2.1} $A$ is separable from $B$ by a $\mathbf\Pi^0_\xi\times\mathbf\Sigma^0_1$ set,
		\item\label{l2.3} $A$ is separable from $B$ by a $(\mathbf\Pi^0_\xi\cap\Delta^1_1)\times(\mathbf\Sigma^0_1\cap\Delta^1_1)$ set,
		\item\label{l2.2} for all $y\in\Pi_1[A]$ there is an open neighborhood $V$ of $y$ such that
		\[(\overline{\Pi_0[A]}^{\tau_\xi}\times V)\cap B=\emptyset.\]
	\end{enumerate}
\end{enumerate}
\end{lem}
\begin{proof} ($1.$)($(a)\Rightarrow (b)$). Suppose that we can separate $A$ from $B$ by a $\close\times\mathbf\Sigma^0_\xi$ set, say $C\times D$. In particular $D$ separates $\Pi_1[A]$ from $\{y\in Y| \exists x\in \overline{\Pi_0[A]} (x,y)\in B\}$. Note that this last set is $\Sigma^1_1$. By \cite[Theorem B]{LO1980}, we obtain a set  $D'\in \mathbf\Sigma^0_\xi\cap\Delta^1_1$ which separates the previous sets. This implies that $\Pi_0[A]$ is separable from $\{x\in X| \exists y\in D' (x,y)\in B\}$ by $\overline{\Pi_0[A]}$. Thus, again by \cite[Theorem B]{LO1980} we obtain a $C'\in\mathbf\Pi^0_1\cap\Delta^1_1$ such that $C'\times D'$ separates A from B.

($(b)\Rightarrow (c)$) Suppose that we can separate $A$ from $B$ by a $(\mathbf\Pi^0_1\cap\Delta^1_1)\times(\mathbf\Sigma^0_\xi\cap\Delta^1_1)$ set, say $C\times D$. In particular $D$ separates $\Pi_1[A]$ from $\{y\in Y| \exists x\in \overline{\Pi_0[A]} (x,y)\in B\}$. Note that this last set is $\Sigma^1_1$. By Lemma \ref{LO1}, $\Pi_1[A]$ is contained in 
\[V=\neg\overline{\{y\in Y| \exists x\in \overline{\Pi_0[A]} (x,y)\in B\}}^{\tau_\xi},\]
which is $\tau_\xi$-open. Therefore, for each $y\in\Pi_1[A]$, $y$ is in $V$ and $\overline{\Pi_0[A]}\times V$ does not intersect $B$, as desired.

($(c)\Rightarrow (a)$). Let $(V_n)_{n\in\omega}$ be a basis for $\tau_\xi$. Then, for each $y\in \Pi_1[A]$, we can find a $n_{y}\in\omega$ such that $V_{n_y}$ witnesses $(c)$. Then:
\[\overline{\Pi_0[A]}\times \bigcup_{y\in\Pi_1[A]}V_{n_y}\]
is a $\close\times\mathbf\Sigma^0_\xi$ that separates $A$ from $B$.

($2.$) Note ($(b)\Rightarrow (a)$) is obvious. We will show ($(a)\Rightarrow (c)$) and ($(c)\Rightarrow (b)$).

($(a)\Rightarrow (c)$). Suppose that we can separate $A$ from $B$ by a $\mathbf\Pi^0_\xi\times\open$ set, say $C\times U$.  Each $y\in \Pi_1[A]$ is in $U$ so there is a a $\semirecursive$ neigborhood $U'$ of $y$ with $U'\subseteq U$. In particular, $(C\times U')\cap B=\emptyset$. This means that $\Pi_0[A]$ is separable from $\{x\in X| \exists y\in U' (x,y)\in B\}$ by $C$, a $\mathbf\Pi^0_\xi$ set. Again, by Lemma \ref{LO1} , we obtain $(\overline{\Pi_0[A]}^{\tau_\xi}\times U')\cap B=\emptyset$.

($(c)\Rightarrow (b)$). Let $(V_n)_{n\in\omega}$ be a basis for $\tau_1$ consisting of $\semirecursive$ sets. Then, for each $y\in \Pi_1[A]$, we can find a $n_{y}\in\omega$ such that $V_{n_y}$ witnesses $(c)$. We claim that we can choose this in a $\Delta^1_1$ way. In fact, ``$V_{n}\cap \Pi_1[(\overline{\Pi_0[A]}^{\tau_\xi}\times Y)\cap B]=\emptyset$'' is a $\Pi^1_1$ property in $n$, so that we can apply the $\Delta^1_1$-selection principle.  Then
\[\overline{\Pi_0[A]}^{\tau_\xi}\times \bigcup_{y\in\Pi_1[A]}V_{n_y}\]
is a $(\mathbf\Pi^0_\xi\cap\Sigma^1_1)\times(\open\cap\Sigma^1_1)$ that separates $A$ from $B$. By applying \cite[Theorem B]{LO1980} twice as before, we can find $C'\times D'\in(\mathbf\Pi^0_\xi\cap\Delta^1_1)\times(\mathbf\Sigma^0_1\cap\Delta^1_1)$ which separates $A$ from $B$.  
\end{proof}

Note that in the case $\xi=1$, we can get a stronger version in the non-effective case. We only require the spaces to be Polish, and  $A,B$ to be disjoint. In fact, if $A$ is separable from $B$ by a $\close\times\open$ set, then, for each $y\in\Pi_1[A]$, take $V$ as the open side that witnesses the separation. 

\section{Separation by a $\open\times\open$ set}\label{sec3}

We we first provide our $\mathcal{C}$ for the problem of separability by a $\open\times\open$ set, since this contrasts a lot with the other results. In particular, this is the only case where our antichain basis has two elements. 

\begin{defi} A \textbf{left-branching scheme} of a zero-dimensional space $\bbX$ is a family of non-empty clopen subsets $\mathcal{F}:=\{F_{m,\varepsilon}| m\in\omega, \varepsilon\in 2\}$ such that
\begin{enumerate}
	\item $\{F_{0,\varepsilon}| \varepsilon\in 2\}$ is a partition of $\bbX$,
	\item $\{F_{m+1,\varepsilon}| \varepsilon\in 2\}$ is a partition of $F_{m,0}$.
\end{enumerate}
We say that it \textbf{converges} if $\mbox{diam}(F_{m,0})\to 0$ as $m\to\infty$.
\end{defi}

Note that if a left-branching scheme converges then $\cap_{m\in\omega} F_{m,0}$ is a singleton. 

In \cite{LZ12}, the authors define, for $\xi<\omega_1$, a \textbf{$\xi$-disjoint family} as a family of sets which are $\mathbf\Pi^0_\xi$ and pairwise disjoint, where by convention $\mathbf\Pi^0_0=\mathbf\Delta^0_1$.

Given a zero-dimensional Polish space $\bbX$, consider a $0$-disjoint partition $\{F_n\}$. In each $F_n$, consider a left-branching scheme $\{F^n_{m,\varepsilon}|m\in\omega, \varepsilon\in 2\}$. Finally, for each $F^0_{m,1}$, consider a $0$-disjoint partition $\{C^{m}_l| l\in\omega\}$. Define the following sets:

\[\bbA:=\Delta(\bigcup_{n\in\omega}\bigcap_{m\in\omega} F^n_{m,0}),\]

\[ \bbB_0:= \bigcup_{n,m\in\omega} \big( C^n_{m}\times F^{n+1}_{m,1}\big),\]
\[ \bbB_1:= \bigcup_{n,m\in\omega} \big( F^{n+1}_{m,1}\times C^n_m\big).\]

\begin{thm}\label{Op}  Let $X,Y$ be Polish spaces, and $A,B$ be disjoint subsets of $X\times Y$. At least one of the following holds:
\begin{enumerate}
	\item[i)] $A$ is separable from $B$ by an open rectangle,
	\item[ii)] there is $i\in 2$ such that $(\bbX,\bbX,\bbA,\bbB_i)\leq (X,Y,A,B)$.
\end{enumerate}
If moreover, the left-branching scheme $\{F^n_{m,e}|m\in\omega, e\in 2\}$ converges for each $n\in\omega$, then exactly one of the previous holds.
\end{thm}

\begin{proof} Suppose that $A$ is not separable from $B$ by an open rectangle. Then, 1. or 2. of Lemma \ref{L2} must hold. Suppose that 1. holds, we will show that $(\bbX,\bbX,\bbA,\bbB_0)$ reduces to $(X,Y,A,B)$. We note that the same argument will give us the reduction for $i=1$ when 2. holds. 

Find $x\in\Pi_0[A]$ which satisfies 1. and let $y\in Y$ such that $(x,y)\in A$. It is easy to see that, for each $n,m\in\omega$, there are $(x_n,y_n)\in A$ and $(x^n_m,y^n_m)\in B$ such that,
\begin{enumerate}
	\item[(a)] for all sequences $(m_n)_{n\in\omega}$, $x^n_{(m_n)}\to x$ as $n\to\infty$,
	\item[(b)] $y^n_m\to y_n$ as $m\to\infty$.
\end{enumerate}

We define $f:\bbX\to X$
\[f(\alpha):=\left\{\begin{array}{ll}
           x_n & \mbox{if } \alpha\in F_{n+1},\\
           x^n_m & \mbox{if } \alpha\in C^n_m\\
           x & \mbox{if } \alpha\in \cap_{l\in\omega} F^0_{l,0}.
          \end{array} \right. 
\]

Similarly, define $g:\bbX\to Y$ by:

\[g(\alpha):=\left\{\begin{array}{ll}
					 y_n & \mbox{if } \alpha\in\cap_{l\in\omega} F^{n+1}_{l,0}\\
           y^n_m & \mbox{if } \alpha\in F^{n+1}_{m,1}\\
           y & \mbox{if } \alpha\in F_0.
          \end{array} \right. 
\]

It is easy to see that these maps are well defined. 

We first show the continuity of $f$. Notice that we only need to check the continuity in $\cap_{l\in\omega} F^0_{l,0}$, since each other part is clopen. So suppose that $\alpha\in\cap_{l\in\omega} F^0_{l,0}$ and $\alpha_k\to\alpha$. For each $l$ there is $K$ such that $\alpha_k\in F^0_{l,0}$ for $k>K$.  We can suppose that $\alpha_k\notin\cap_{l\in\omega} F^0_{l,0}$, since otherwise $f(\alpha_k)=x$. Then $f(\alpha_k)=x^{l_k}_{m_k}$, for some $l_k$ greater than $l$. 

We claim that $l_k$ diverges. Indeed, $l_k\geq l$ if $\alpha_k\in F^0_{l,m}$, and for each $l$ this holds for $k$ big enough. Then, $x^{l_k}_{m_k}\to x$, as required. To check continuity for $g$, we only need to verify it in $\cap_{l\in\omega} F^{n+1}_{l,0}$, and this is done similarly. Checking that $f\times g$ is a reduction is routine.

Suppose that each scheme converges. We will show that $\bbA$ is not separable from $\bbB_0$. So let $U, V$ be open subsets such that $\bbA\subseteq U\times V$.  Note that $\cap F^0_{l,0}=\{x\}$, and $\{F^0_{l,0}\}$ is a basis at $x$ made of clopen neighborhoods. Then, there is an $l$ such that $F^0_{l+1,0}\subseteq U$. Note that $\cap_{m\in\omega} F^{l+2}_{m,0}=\{y_l\}\subseteq V$, so there is an $m$ such that $F^{l+2}_{m,0}\subseteq V$.  Then 
$C^{l+1}_{m+1}\times F^{l+2}_{m+1,1}\subseteq \bbB_0\cap \big(F^0_{l,0}\times F^{l+2}_{m,0}\big)\subseteq \bbB_0\cap \big(U\times V\big)$. 
Thus, $\bbA$ is not separable from $\bbB$ by an open rectangle. 

This shows that at most 1. or 2. must hold, as in the proof of Proposition \ref{Borel}.
\end{proof}

One can obtain an example satisfying our conditions in $\baire$ by setting $F_n=N_{(n)}$,
\[F^n_{m,\varepsilon}=\left\{\begin{array}{ll}
					 N_{n^{m+2}} & \mbox{if } \varepsilon=0,\\
           \bigcup_{k\neq n}N_{n^{m+1}k} & \mbox{if } \varepsilon=1.
          \end{array}\right.\]
and $C^m_l=N_{0^{m+1}(l+1)}$. It is routine to show they satisfy the hypothesis in the construction of $\bbA$ and $\bbB$, and that that each scheme is converging. 					
					
We can in fact shrink a bit our minimal examples, as some easy verification can show they stay non separable.

\[\mathbb{A}:=\{(n^\infty,n^\infty)| n\in \omega\},\]
\[ \mathbb{B}_0:= \{(0^{n+1}{(m+1)}^\infty,{(n+1)}^{m+1}0^\infty)| n,m\in \omega\},\]
\[ \mathbb{B}_1:= \{({(n+1)}^{m+1}0^\infty,0^{n+1}{(m+1)}^\infty)| n,m\in \omega\},\]

We would also note that this is the best we can do.

\begin{Pro} $(\baire,\baire,\bbA,\bbB_0)$ and $(\baire,\baire,\bbA,\bbB_1)$ are $\leq$-incomparable. 

\end{Pro}
\begin{proof}
  Suppose that $(\baire,\baire,\bbA,\bbB_0)\leq(\baire,\baire,\bbA,\bbB_1)$, the other side being similar. Then, for any $n,m$ there are $k_{n,m}$ and $l_{n,m}$ such that $f(0^{n+1}{(m+1)}^\infty)=(k_{n,m}+1)^{l_{n,m}+1}0^\infty$, and $g({(n+1)}^{m+1}0^\infty)=0^{k_{n,m}+1}(l_{n,m}+1)^\infty$.
	
As $f(0^{n+1}{(m+1)}^\infty)\to f(0^\infty)$, when $n\to\infty$, $f(0^\infty)=(K+1)^\infty$ for some $K\in\omega$. Then, there is $N\in\omega$, such that for $n\geq N$ and any $m\in\omega$, $f(0^n(m+1)^\infty)\in N_{K+1}$, by continuity. However $0^{K+1}(l_{n,m}+1)^\infty=g({(n+1)}^{m+1}0^\infty)$, but as $m\to\infty$, the left side cannot converge to something of the form $(K')^\infty$ and the right side converges to something of this form. 
\end{proof}

\section{Separation by a $\close\times\open$ set}

This case has some things in common with the open case. In particular, the construction uses the same families of sets. Let $\{F_n\}$ be a $0$-disjoint partition of $\bbX$, and, for each $n$, $\{F^n_{m,\varepsilon}\}$ be a left-branching scheme. Instead of the $0$-disjoint partitions of $F^0_{m,1}$, we consider for each $n\in \omega$, a $0$-disjoint partition $\{D^n_m|m\in\omega\}$ of $F_{n+1}$. We now define the following sets:

\[\bbA:=\big((\cap_{l\in\omega} F^0_{l,0})\times (\cap_{l\in\omega}  F^0_{l,0})\big)\bigcup \big(\cup_{m,n\in\omega}(F^{n+1}_{m,1}\times D^{n}_m)\big),\]

\[\bbB:= \bigcup_{n\in\omega}\big((\cap_{l\in\omega} F^{n+1}_{l,0}) \times F^0_{n,1} \big).\]

\begin{Pro} Let $X,Y$ be Polish spaces, and $A,B$ be disjoint subsets of $X\times Y$. At least one of the following holds:
\begin{enumerate}
	\item $A$ is separable from $B$ by a $\close\times\open$ set,
	\item $(\bbX,\bbX,\bbA,\bbB)\leq (X,Y,A,B)$.
\end{enumerate}
If moreover each scheme converges, then exactly of the previous holds.
\end{Pro}
\begin{proof} Suppose that $A$ is not separable from $B$ by a $\close\times\open$ set. 

By Lemma \ref{L3}, there is $y\in\Pi_1[A]$ such that for each open neighborhood $V$ of $y$, 
$(\overline{\Pi_0[A]}\times V)\cap B\neq\emptyset$. Let $x\in X$ with $(x,y)\in A$.

Let $(V_n)$ be a decreasing neighborhood basis at $y$. For each $n\in\omega$, choose 
$(x_n,y_n)\in(\overline{\Pi_0[A]}\times V_n)\cap B$.

We note that $y_n\to y$ as $n\to\infty$. Also, since $x_n\in\overline{\Pi_0[A]}$, there is a sequence $(x^n_m,y^n_m)\in A$, such that $x^n_m\to x_n$ as $m\to\infty$.

Define the following functions:
\[\begin{array}{ll}
f(\alpha)&=\left\{\begin{array}{ll}
	x & \mbox{if } \alpha\in F_0,\\
	x_n & \mbox{if } \alpha\in \cap_{l\in\omega} F^{n+1}_{l,0},\\
	x^n_m & \mbox{if } \alpha\in  F^{n+1}_{m,1}.
									\end{array}\right.\\
									& \\
g(\alpha)&=\left\{\begin{array}{ll}
	y & \mbox{if } \alpha\in \cap_{l\in\omega} F^0_{l,0},\\
	y_n & \mbox{if } \alpha\in F^0_{n,1},\\
	y^n_m & \mbox{if } \alpha\in D^{n}_m.
\end{array}\right.
\end{array}\]

The proof that this is in fact a well defined reduction is the same as that of Theorem \ref{Op}. By Lemma \ref{L3}, if the scheme converges, $\bbA$ is not separable from $\bbB$ by a $\close\times\open$ set.
\end{proof}

Using the same families as in the open case, one can obtain a more concrete example. Choose also, $D^{n}_m=N_{(n+1)m}$. We shrink the examples as in the previous case, to get the following sets, which define a $\leq$ minimal example for non separation by $\close\times\open$ set.

\[\bbA:=\{(0^\infty,0^\infty)\}\cup\{\big((n+1)^{m+1}0^\infty,(n+1)m^\infty\big)|n,m\in\omega\},\]
and 
\[\bbB:=\{\big((n+1)^\infty,0^{n+1}1^\infty\big)|n\in\omega\}.\]

\section{Separation by a $\mathbf\Pi^0_\xi\times\mathbf\Pi^0_\xi$ set}

Lecomte and Zeleny used Lemma \ref{LZ1} to prove Conjecture \ref{LZ2} in the cases $\xi=1,2$. As one can expect from the proof of Lemma \ref{lem1}, this conjecture provides a weaker dichotomy for the case of one $\mathbf\Pi^0_\xi$ rectangle. 

\begin{lem}\label{Lem1} If Conjecture \ref{LZ2} is true, then for every $0<\xi<\omega_1$, for every Polish spaces $X,Y$, and for every disjoint analytic  $A,B\subseteq X\times Y$, exactly one of the following holds:
\begin{enumerate}
	\item $A$ is separable from $B$ by a $\mathbf\Pi^0_\xi\times \mathbf\Pi^0_\xi$ set,
	\item there are continuous functions $f:\mathbb{X}_\xi\to X$ and $g:\mathbb{Y}_\xi\to Y$ such that
	\begin{align*}
	\Pi_0[\mathbb{A}_\xi]&\subseteq f^{-1}(\Pi_0[A]),\\
	\Pi_1[\mathbb{A}_\xi]&\subseteq g^{-1}(\Pi_1[A]),\\
	\mathbb{B}_\xi&\subseteq (f\times g)^{-1}(B).
	\end{align*}
\end{enumerate}
\end{lem}
\begin{proof} For the exactly part, we note that if $C\times D$ separates $A$ from $B$, with $C,D$ in $\mathbf\Pi^0_\xi$, then $f^{-1}[C]\times g^{-1}[D]$ separates $\Pi_0[\mathbb{A}_\xi]\times\Pi_1[\mathbb{A}_\xi]$ from $\bbB_\xi$, which cannot be the case, because of Conjecture \ref{LZ2}.
 
Now, by relativization, we can suppose that $X,Y$ are recursively presented Polish spaces, and that $A,B$ are $\Sigma^1_1$ subsets of $X\times Y$. Suppose that $A$ is not separable from $B$ by a $\mathbf\Pi^0_\xi\times \mathbf\Pi^0_\xi$ set. Lemmas \ref{lem1}, \ref{LZ1}, and Conjecture \ref{LZ2} give our functions. 
\end{proof}

We can improve this lemma by finding an actual $\leq$-minimum example for the cases $\xi=1,2$, which coincidentally are the cases where Conjecture \ref{LZ2} is proved.  In both cases, we need an additional hypothesis on $\mathbb{A}_\xi$, which can be easily fulfilled, as we will see below. First, we introduce some notation

\begin{nota} Let $X,Y$ be topological spaces. Then $X\oplus Y$ denote the topological sum of $X,Y$, i.e., the set $\{(\varepsilon,z)\in 2\times (X\cup Y)| (\varepsilon=0\land z\in X)\lor (\varepsilon=1\land z\in Y)\}$, with the smallest topology where each copy of $X$ and $Y$ is clopen. When there is no chance of confusion, we will denote by $\overline{z}=(\varepsilon,z)$. In particular, we will use variables $\alpha,\beta,\gamma,...$ for elements of a general topological space $X$, and variables $n,m,...$ for elements of $\omega$. 
\end{nota}

Let $\xi\in\{1,2\}$. Suppose that $(\bbX_\xi,\bbY_\xi,\bbA_\xi,\bbB_\xi)$ satisfy Conjecture \ref{LZ2}. Suppose that $\bbA_\xi$ has countable projections, so let $\{\alpha_n\}_{n\in\omega}$ and $\{\beta_n\}_{n\in\omega}$ be enumerations of the first and second projections respectively. We define 
\begin{align*}
\bbX'_\xi&:=\bbX_\xi\oplus\omega,\\
\bbY'_\xi&:=\bbY_\xi\oplus\omega, \\ 
\bbA'_\xi&:=\big\{(\overline{\alpha_n},\overline{n})\in\bbX'_\xi\times\bbY'_\xi| n\in\omega\}\cup\{(\overline{n},\overline{\beta_n})\in\bbX'_\xi\times\bbY'_\xi|n\in\omega\big\},\\
\bbB'_\xi&:=\big\{(\overline{\alpha},\overline{\beta})\in\bbX'_\xi\times\bbY'_\xi| (\alpha,\beta)\in \bbB_\xi\big\}.
\end{align*}
We obtain the following theorem.

\begin{thm} \label{Thm2} Let $\xi\in\{1,2\}$. Let $X,Y$ be Polish spaces, and $A,B\subseteq X\times Y$ be disjoint analytic subsets. Exactly one of the following holds:
\begin{enumerate}
	\item $A$ is separable from $B$ by a $\mathbf\Pi^0_\xi\times\mathbf\Pi^0_\xi$ set,
	\item $(\bbX'_\xi,\bbY'_\xi,\bbA'_\xi,\bbB'_\xi)\leq(X,Y,A,B)$.
\end{enumerate} 
\end{thm}
\begin{proof} Applying the previous lemma to $\bbA'_\xi$ and $\bbB'_\xi$ we can see that $\bbA'_\xi$ is not separable from $\bbB'_\xi$ (using the canonical embeddings from $\bbX_\xi$ into $\bbX'_\xi$, and $\bbY_\xi$ into $\bbY'_\xi$), so the exactly part follows.

Suppose that $A$ is not separable from $B$ by a $\mathbf\Pi^0_\xi\times\mathbf\Pi^0_\xi$ set. Lemma \ref{Lem1} gives auxiliary continuous functions $f':\bbX_\xi\to X$ and $g':\bbY_\xi\to Y$. Now, $f'(\alpha_n)\in\Pi_0[A]$, so there is $y_n\in Y$ such that $(f'(\alpha_n),y_n)\in A$. Similarly, we can obtain $x_n\in X$, such that $(x_n,g'(\beta_n))\in A$. We then define $f:\bbX'_\xi \to X$, and $g:\bbY'_\xi \to Y$ by:

\begin{align*}
		f(\overline{\alpha})&=f'(\alpha)\\
		f(\overline{n})&=x_n
\end{align*}

and 
\begin{align*}
		g(\overline{\beta})&=g'(\beta)\\
		g(\overline{n})&=y_n
	\end{align*}
	
These are continuous maps, and $f\times g$ is a reduction, by Lemma \ref{Lem1} and the choice of $x_n$ and $y_n$.
\end{proof}

In order to be more concrete, we will like to give particular instances for these $\leq$-minimum examples. These can be obtained directly from \cite{LZ12}. We recall the general form of these examples, and then give a particular example that satisfy our condition on the projections of $\bbA_\xi$. 

\begin{Pro} \begin{enumerate}
	\item (Lecomte-Zeleny) Let $\bbX$ and $\bbY$ be $0$-dimensional Polish spaces, and let $\{C^0_i|i\in\omega\}$ and $\{C^1_i|i\in\omega\}$ be 0-disjoint families of $\bbX$ and $\bbY$ respectively.  Then if 
	$\bbB\subseteq\big(\bbX\backslash\big(\bigcup_{i\in \omega}C^0_i\big)\big)\times\big(\bbY\backslash\big(\bigcup_{i\in \omega}C^1_i\big)\big)$ 
	is not separable from $\bbA\subseteq\bigcup_{i\in\omega} C^0_i\times C^1_i$  by a $(\open\times\open)_\sigma$ set, then they satisfy Conjecture \ref{LZ2} for $\xi=1$.
	\item There are $\bbX$, $\bbY$, $\bbA$ and $\bbB$ which satisfy the hypothesis in 1.  such that $\bbA$ has countable projections.
\end{enumerate}
\end{Pro}

\begin{proof} Let $\bbX:=\bbY:=2^\omega$. It is easy to see that $C^\varepsilon_n:=N_{0^n1}$ defines a $0$-disjoint  family. Then $\bbB:=\{(0^\infty,0^\infty)\}$ and $\bbA:=\{(0^n1^\infty,0^n1^\infty);n\in\omega\}$ satisfy 2.
\end{proof}

We obtain from this Proposition and Theorem \ref{Thm2} the following minimum example for non separability by closed rectangles:

\[\bbA'=\{(\overline{0^n1^\infty},\overline{n})| n\in \omega\}\cup \{(\overline{n},\overline{0^n1^\infty})| n\in \omega\},\] \[\bbB=\{(\overline{0^\infty},\overline{0^\infty})\}.\] 

We now consider $\gdelta$ rectangles. The following definition was introduced in \cite{LZ12}.

\begin{defi} (Lecomte-Zeleny) Let $1\leq\xi<\omega_1$. A $\xi$-disjoint family $(C^\varepsilon_i)_{(\varepsilon,i)\in2\times\omega}$ of subsets of a 0-dimensional Polish space $\mathbb{W}$ is said to be \textbf{very comparing} if for each natural number q, there is a partition $(O^p_q)_{p\in\omega}$ of $\mathbb{W}$ into $\mathbf\Delta^0_\xi$ sets such that, for each $i\in\omega$,
\begin{enumerate}
	\item if $q<i$, then there is a $p^i_q\in\omega$ such that $C^0_i\cup C^1_i\subseteq O^{p^i_q}_q$,
	\item if $q\geq i$ and $\varepsilon\in 2$, then $C^\varepsilon_i\subseteq C^{2i+\varepsilon}_q$,
	\item if $(\varepsilon,i)\in 2\times\omega$, then $\bigcup_{r\geq i}\bigcap_{q\geq r} O^{2i+\varepsilon}_q=C^\varepsilon_i$.
\end{enumerate}
\end{defi}

\begin{Pro}\label{LZE}\begin{enumerate}
	\item (Lecomte-Zeleny) Let $(C^\varepsilon_i)_{(\varepsilon,i)\in2\times\omega}$ be a very comparing 1-disjoint family of subsets of a $0$-dimensional Polish space $\mathbb{W}$. Let  
	$\bbX\subseteq\mathbb{W}\backslash\bigcup_{i\in \omega}\big(C^1_i\big), \bbY\subseteq\mathbb{W}\backslash\big(\bigcup_{i\in \omega}C^0_i\big),\bbB\subseteq \Delta(\mathbb{W}\backslash \big(\bigcup_{(\varepsilon,i)\in 2\times\omega} C^\varepsilon_i\big))$, and $\bbA\subseteq \bigcup_{i\in\omega} C^0_i\times C^1_i$
 such that $\bbB$ is not separable from $\bbA$ by a $(\fsigma\times\fsigma)_\sigma$ set. Then they satisfy Conjecture \ref{LZ2} for $\xi=2$.
	\item There are $\bbX$, $\bbY$, $\bbA$ and $\bbB$ which satisfy the hypothesis in 1.  such that $\bbA$ has countable projections.
\end{enumerate}
\end{Pro}
\begin{proof} Take $\mathbb{W}:=3^\omega$, and $C^\varepsilon_i=\{\theta(i)\varepsilon\alpha|\alpha\in 2^\omega\}$, where $\theta$ is an enumeration of $\{s\in 3^{<\omega}|s=\emptyset \lor s(|s|-1)=2\}$. This was shown in \cite{LZ12} to be a very comparing  family. By taking $\bbB:=\Delta\big(\mathbb{W}\backslash  \big(\bigcup_{(\varepsilon,i)\in 2\times\omega}C^\varepsilon_i\big)\big)$, and $\bbA:=\{(\theta(i)0^\infty,\theta(i)1^\infty)|i\in\omega\}$, we obtain 2.
\end{proof}

Then we obtain our examples for non separability by $\gdelta$ rectangles:

\[\bbA'=\{(\overline{\theta(n)0^\infty},\overline{n})|n\in\omega\}\cup\{(\overline{n},\overline{\theta(n)1^\infty})|n\in\omega\},\]
\[\bbB'=\{(\overline{\alpha},\overline{\alpha})|\alpha\in\bbB\}.\]

\section{Separation by  a $(\mathbf\Sigma^0_1\times\mathbf\Sigma^0_2)_\sigma$ set and by a $\mathbf\Pi^0_1\times\mathbf\Pi^0_2$ set}

One would like to know if the argument in the previous section extends to other classes of the type $\Gamma\times\Gamma'$. Even if Conjecture \ref{LZ2} holds for $\xi>2$, the first that is clear is that if $\Gamma=\Gamma'=\mathbf\Pi^0_\xi$, then we cannot extend the definition right before Theorem \ref{Thm2}. Indeed, if $\bbA$ has countable projections,  then the product of its projections will be a $\mathbf\Pi^0_\xi\times\mathbf\Pi^0_\xi$ set which separates $\bbA$ from $\bbB$.

So another option is to see what happens if $\Gamma\neq\Gamma'$. In particular, we can see what happens in the case $\close\times\gdelta$. Following the same argument as in the previous section, it is enough to prove the natural extension of Conjecture \ref{LZ2} for $(\open\times\fsigma)_\sigma$. 

In order to do this, we will work with the Gandy-Harrington topology on a recursively presented Polish space $Z$, hereafter denoted by $GH_Z$. This is the topology generated by all $\Sigma^1_1$ subsets of $Z$. This topology is not regular, however, it is in fact Polish on $\Omega=\{x\in Z| \omega^z_1=\omega^{CK}_1\}$. In fact, $\Omega$ is $\Sigma^1_1$, dense in $(Z,GH_Z)$ and for every $W\in \Sigma^1_1$, $W\cap\Omega$ is a $GH$-clopen subset of $\Omega$.

For each $\fsigma$ subset $F$ of $\cantor$, let $(F_n)$ be an increasing sequence of closed sets such that $F=\cup_{n\in\omega}F_n$. For each $s\in 2^{<\omega}$ let $n_s:=\mbox{min}\{n| F_n\cap N_s\neq\emptyset\}$, when this exists. One can in fact suppose $F$ is dense, so that $n_s$ will always exist. Let $D:=\{s\in 2^{<\omega}|s=\emptyset \lor n_s\neq n_{s^m}\}$, where $s^m:=s|_{|s|-1}$ for $s\neq \emptyset$ and $s^m:=s$ if $s=\emptyset$. With this in mind, set:

\[\bbX_F:= 3^\omega\backslash F,\; \bbY_F:=\cantor,\]
\[\bbB_F:=\Delta(\cantor\backslash F),\]
\[\bbA_F:=\{(\alpha,\beta)| \exists s\in D(s2\sqsubseteq\alpha \land  s\sqsubseteq \beta \land \beta\in F_{n_s})\}.\]

\begin{thm}\label{Thm1} Let $F$ be as above, $X,Y$ be Polish spaces, and $A,B$ be disjoint analytic subsets of $X\times Y$. At least one of the following holds:
\begin{enumerate}
	\item $B$ is separable from $A$ by a $(\open\times\fsigma)_\sigma$ set.
	\item $(\bbX_F,\bbY_F,\bbB_F,\bbA_F)\leq (X,Y,B,A)$
\end{enumerate}
If moreover $F$ is meager in $\cantor$, then exactly one of the previous must hold. 
\end{thm}

\begin{proof} 

Suppose that $B$ is not separable from $A$ by a $(\open\times\fsigma)_\sigma$ set.

Note that, without loss of generality, we can suppose that $X,Y$ are recursively presented Polish spaces, and that $A,B$ are $\Sigma^1_1$. Then, by Lemma \ref{LZ1},  $N:=B\cap\overline{A}^{\tau_1\times\taudos}$ is not empty. 

For each $s\in2^{<\omega}\backslash \{\emptyset\}$, let $s^-:=s|_{\max\{|t|\;| t\sqsubset s \land t\in D\}}$, so that $s^-$ is either empty or the last proper initial segment $t$ of $s$ where $n_{t}$ changed value, and set $\emptyset^-:=\emptyset$.

We construct, for each $t\in 3^{<\omega}$ and $ s\in 2^{<\omega}$, 
\begin{enumerate}
	\item[i.] points $x_t\in X$, $y_s\in Y$,
	\item[ii.] $\semirecursive$ sets $X_t\subseteq X$ and $Y_s\subseteq Y$, 
	\item[iii.] a $\Sigma^1_1$ set $V_{s}\subseteq X\times Y$ if $s\in D$.
\end{enumerate}

We will require these sets to satisfy the following conditions, for each $t\in 3^{<\omega}$ and $s\in 2^{<\omega}$:
\begin{enumerate}
	\item\label{c1} $x_t\in X_t$, $y_s\in Y_s$, and $(x_s,y_s)\in V_s$
	\item\label{c2} $\mbox{diam}(X_t)<2^{-|t|}$, $\mbox{diam}(Y_s)<2^{-|s|}$; and $\mbox{diam}_{GH}(V_{s})<2^{-|s|}$
	\item\label{c3} $\left\{\begin{array}{ll}
	\overline{X_s}\subseteq X_{s^-} &\mbox{if } s\in D\backslash\{\emptyset\}\\
	\overline{X_{t}}\subseteq X_{t^m} &\mbox{if } t\in 3^\omega\backslash D\\
	\overline{Y_{s\varepsilon}}\subseteq Y_s & \mbox{if } \varepsilon\in 2\\
	V_{s}\subseteq V_{s^-} &\mbox{if } s\in D\end{array}\right.$
	\item\label{c4} $V_{s}\subseteq N\cap\Omega\cap(X_{s}\times Y_{s})$
	\item\label{c5} if $s\in D$ and $s\varepsilon\notin D$, then $\left\{\begin{array}{l}(x_{s2},y_{s\varepsilon})\in A\cap\Omega\cap (X_s\times Y_s)\\ y_{s\varepsilon}\in\overline{\Pi_1[V_s]}\end{array}\right.$
	\item\label{c6} $x_{s2t}=x_{s2}$, and, if $s\notin D$, and $s^m\notin D$, then $y_{s}=y_{s^m}$.
\end{enumerate}

So suppose that these are already constructed. We will construct our functions $f:\bbX_F\to X$ and $g:\bbY_F\to Y$. If $\alpha\in 2^\omega$ then we define
 \[\{g(\alpha)\}=\bigcap_{n\in\omega}\overline{Y_{\alpha|_n}}\]
and as usual the map $g$ defined this way is continuous, and $y_{\alpha|_n}\to g(\alpha)$.

For $f$, there are two cases. If there is $i\in\omega$ such that $\alpha(i)=2$, let $i_0$ be the smallest. By conditions \ref{c2} and \ref{c3}, one can define $\{f(\alpha)\}:=\cap_{n>i_0} \overline{X_{\alpha|_n}}$, and $x_{\alpha|_n}\to f(\alpha)$. 

Otherwise, $(n_{\alpha|_i})$ diverges, since otherwise $\alpha\in F$. In this case, by conditions \ref{c2} and \ref{c3}, we can define $\{f(\alpha)\}:=\cap_{i\in\omega} \overline{X_{(\alpha|_{i})^-}}$. 

We note that on the open set $\{\alpha| \exists i (\alpha(i)=2)\}$  the function will be continuous as usual. So suppose $(n_{\alpha|i})$ diverges.  Let $(\alpha_k)$ be a sequence in $\bbX_F$ which converges to $\alpha$. 
Given an open neighborhood of $f(\alpha)$, we can find $s\sqsubseteq\alpha$ such that $s\in D$ and $X_s$ is contained in such a neighborhood. 
Now, for big enough $k$, $s\sqsubseteq \alpha_k$. If $n_{\alpha_k|_i}$ changes infinitely often, then by definition $f(\alpha_k)\in X_{s}$. 
If it does not changes infinitely often, there is a last $k_0$ such that $s\sqsubseteq \alpha_k|_{k_0}$ and $\alpha_k|_{k_0}\in D$. In particular, $X_{\alpha_k|_l}\subseteq X_s$ if $l\geq k_0$, and then, applying condition \ref{c3} and the definition of $f(\alpha_k)$,  $f(\alpha_k)\in X_{s}$.

Now, we need to show that $f\times g$ is a reduction. If $(\alpha,\alpha)\in \bbB_F$, then $n_{\alpha|_i}$ changes infinitely often. Define $\{F(\alpha)\}=\cap V_{(\alpha|_{i})^-}\subseteq B$. By condition \ref{c1}, $(x_{(\alpha|_{k})^-},y_{(\alpha|_{k})^-})\to F(\alpha)$ in the Gandy-Harrington topology. Since this topology refines the product topology, $\big(f(\alpha),g(\alpha)\big)=F(\alpha)\in B$. 

If $(\alpha,\beta)\in \bbA_F$, then condition \ref{c6} and the definition of our maps imply that there is an $s\in D$ such that $\big(f(\alpha),g(\beta)\big)=(x_{s2},y_{s\varepsilon})$, which belongs to $A$ by condition \ref{c5}.

Now, we will show that the construction is possible. We will construct $x_t,y_s, X_t,Y_s,V_s$ by induction on the lengths of $t$ and $s$. For the length $0$, since $N$ is a non-empty $\Sigma^1_1$ set, $N\cap\Omega\neq\emptyset$, there is $(x_\emptyset,y_\emptyset)\in N\cap\Omega$. We find $\semirecursive$ neighborhoods of good diameter, $X_\emptyset$ and $Y_\emptyset$. As 
\[(x_\emptyset,y_\emptyset)\in N\cap\Omega\cap (X_\emptyset\times Y_\emptyset),\]
and this set is $\Sigma^1_1$, we can find another $\Sigma^1_1$ neighborhood $V_\emptyset$ with good diameter in the Gandy-Harrington topology, such that $V_\emptyset\subseteq N\cap\Omega\cap (X_\emptyset\times Y_\emptyset)$. They satisfy the required conditions.

So suppose that we have constructed $x_t, y_s, X_t,Y_s, V_{s}$ for all $s\in 2^p,t\in 3^p$. For each $k\in 3$ and $\varepsilon\in 2$, we will construct the respective points and sets for each finite sequences $tk,s\varepsilon$ of length $p+1$ by cases.

The first case is when $t\notin2^{<\omega}$. In this case, by condition \ref{c6}, we must simply copy the point $x_{tk}:=x_t$ and shrink the $\semirecursive$ neighborhood $X_t$ to a new one of good diameter and which satisfies $\overline{X_{tk}}\subseteq X_t$.

The second case is when $s\in D$. Define $y_{s\varepsilon}=y_s$ if $s\varepsilon\in D$ and $x_{s\varepsilon}=x_s$ if $s\in 2^{<\omega}$. Note also that $(x_{s},y_{s})\in (X_{s}\times Y_{s})\cap (X\times\overline{\Pi_1[V_{s}]})\cap N$. In particular $(X_{s}\times  Y_{s})\cap(X\times\overline{\Pi_1[V_{s}]})\cap \overline{A}^{\tau_1\times\taudos}\neq\emptyset$. So there exists $(x,y)\in (X_{s}\times Y_{s})\cap (X\times\overline{\Pi_1[V_{s}]})\cap A\cap\Omega$. Define $x_{s2}=x$ and $y_{s\varepsilon}=y$ if $s\varepsilon\notin D$. 

In all of these cases $x_{sk}\in X_s$ and $y_{s\varepsilon}\in Y_s$, so shrink the neighborhoods $X_{s}$ and $Y_{s}$ to $\Sigma^0_1$ neighborhoods $X_{sl}$ and $Y_{sl}$ of good diameter. Finally, if $s\varepsilon\in D$, since $(x_{s\varepsilon},y_{s\varepsilon})\in V_s\cap( X_{s\varepsilon}\times Y_{s\varepsilon})\cap \Omega$, we can also shrink $V_s$ to a new $\Sigma^1_1$ set $V_{s\varepsilon}\subseteq V_s\cap (X_{s\varepsilon}\times Y_{s\varepsilon})$ of good diameter. It is easy to check that these elements and sets satisfy all our conditions.

The third and final case is when $s\in 2^{<\omega}\backslash D$. If $s\varepsilon\notin D$, we can define $x_{s\varepsilon}=x_s$. If  $\varepsilon\neq 2$, let $y_{s\varepsilon}=y_s$. Shrink $X_s$ and $Y_s$ accordingly. Clearly, if $s\varepsilon\in 2^{<\omega}$, then condition \ref{c6} is satisfied. 

Since $s\in 2^{<\omega}$, note that $y_{s}=y_{s^-\varepsilon'}\in \overline{\Pi_1[V_{s^-}]}\cap Y_s$ for $\varepsilon'\in 2$ such that $s^-\varepsilon'\subseteq s$. Then, $\Pi_1[V_{s^-}]\cap Y_s\neq \emptyset$. So take $(x,y)\in V_{s^-}\cap (X_{s^-}\times Y_s)\cap \Omega$. If $s\varepsilon\in D$, define $x_{s\varepsilon}=x$ and $y_{s\varepsilon}=y$. Then shrink the neighborhoods $X_{s^-}$ and $Y_s$ to new neighborhoods $X_{s\varepsilon}$ and $Y_{s\varepsilon}$. Finally find a $\Sigma^1_1$ neighborhood $V_{s\varepsilon}\subseteq V_{s^-}\cap (X_{s\varepsilon}\times Y_{s\varepsilon})\cap \Omega$. It is clear that these objects satisfy our conditions. 

So the construction is possible. 

Now suppose that $F$ is meager. We will show that $\bbB_F$ is not separable from $\bbA_F$, so that we can only have at most one of our options, like in previous cases. 

So suppose $\bbB_F\subseteq \cup_{i\in\omega} (U_i\times V_i)$, with $U_i\in\open(3^\omega\backslash F)$ and $V_i\in\fsigma(2^\omega)$. In fact, we can suppose $V_i\in\close(2^\omega)$, for each $i\in\omega$. First, note that $\cantor\backslash F\subseteq \cup_{i\in\omega} (U_i\cap V_i)$. Since $F$ is meager, there is $i$ such that $U_i\cap V_i$ is not meager in $\cantor$. So there exists $s\in 2^{<\omega}$ such that $N_s(2^\omega)\subseteq U_i\cap V_i$.

 Note that there is 
$s\sqsubseteq \hat{s}\in 2^{<\omega}$ such that $N_{\hat{s}}(3^\omega\backslash F)\subseteq U_i$. 
Thus, if there is no $\hat{s}\sqsubseteq s'\in D$, 
then $N_{\hat{s}}(2^\omega)\subseteq F_{n_{\hat{s}}}$ since $F_{n_{\hat{s}}}$ is closed. This contradicts the fact that $F$ is meager. 
Finally, take $\alpha\in N_{s'2}(3^\omega\backslash F)$ and $\beta\in F_{n_{s'}}\cap N_{s'}(2^\omega)$, 
so that $(\alpha,\beta)\in (U_i\times V_i)\cap \bbA_F$. 
\end{proof}

\begin{Pro}\label{Pro4}\begin{enumerate}
	\item  Let $F_n$ be closed subsets of $2^\omega$ such that their union $F$ is dense. If  
	$\bbX\subseteq\bbX_F, \bbY\subseteq\bbY_F,\bbB\subseteq \bbB_F$, $\bbA\subseteq \bbA_F$, and
 $\bbB$ is not separable from $\bbA$ by a $(\open\times\fsigma)_\sigma$ set, then $\bbX,\bbY,\bbA$ and $\bbB$ satisfy our previous theorem.
	\item There are $\bbX$, $\bbY$, $\bbA$ and $\bbB$ which satisfy the hypothesis in 1.  such that $\bbA$ has countable projections.
\end{enumerate}
\end{Pro}
\begin{proof}
Let $\Psi:n\to \{s\in 2^{<\omega}|s=\emptyset \lor s(|s|-1)=1\}$ a bijection such that $\Psi^{-1}(s)\leq \Psi^{-1}(t)$ if $s\subseteq t$. Take $F_n=\{\Psi(m)0^\infty| m\leq n\}$. In particular $D=\{\Psi(m)|m\in\omega\}$. We obtain the following sets:

\[\bbX:=3^\omega\backslash \{\alpha\in\cantor; \forall^\infty i, \alpha(i)=0\}, \bbY:=\cantor\]
\[\bbB:=\{(\alpha,\alpha)\in \bbX\times\bbY| \exists^\infty i \;\alpha(n)=1\}\]
\[\bbA:=\{(\Psi(m)2\alpha,\Psi(m)0^\infty)| \; m\in\omega \land \alpha\in 3^\omega\}\]

One can then define $\bbA=\{(\Psi(m)20^\infty,\Psi(m)0^\infty)| \; m\in\omega\}$. It is routine to check they remain non separable. 
\end{proof}

\begin{Cor}\label{cor1} Let $X,Y$ be Polish spaces, and $A,B$ be disjoint analytic subsets of $X\times Y$, exactly one of the following holds:
\begin{enumerate}
	\item $A$ is separable from $B$ by a $\close\times\gdelta$ set,
	\item there are continuous functions $f:\bbX\to X$ and $g:\bbY\to Y$ such that:
	\begin{align*}
	\Pi_0[\bbA]&\subseteq f^{-1}(\Pi_0[A]),\\
	\Pi_1[\bbA]&\subseteq g^{-1}(\Pi_1[A]),\\
	\bbB & \subseteq (f\times g)^{-1}(B).
	\end{align*}
\end{enumerate}
\end{Cor}

If $\bbX',\bbY',\bbA',\bbB'$ are defined exactly as in the previous section, then we copy the same proof to get the following Corollary.

\begin{Cor}\label{cor2} Let $X,Y$ be Polish spaces, and $A,B$ be disjoint analytic subsets of $X\times Y$. Exactly one of the following holds:
\begin{enumerate}
	\item $A$ is separable from $B$ by a $\close\times\gdelta$ set,
	\item $(\bbX',\bbY',\bbA',\bbB')\leq (X,Y,A,B)$.
\end{enumerate} 
\end{Cor}

So our antichain basis $\mathcal{C}$ is defined by the following objects.
\begin{align*}
\bbX':=&\bbX\oplus\omega,\\
\bbY':=&\bbY\oplus\omega, \\ 
\bbB':=&\big\{(\overline{\alpha},\overline{\alpha})\in\bbX'\times\bbY'| \exists^\infty j (\alpha(j)\neq 0)\big\}, \\
\bbA':=&\big\{(\overline{\Psi(m)20^\infty},\overline{m})\in\bbX'\times\bbY'| m\in\omega\}\cup \\
&\{(\overline{m},\overline{\Psi(m)0^\infty})\in\bbX'\times\bbY'|m\in\omega\big\}.
\end{align*}

\section{Separation by a $\mathbf\Pi^0_2\times\mathbf\Sigma^0_1$ set}

This is an interesting case, as it combines must of the methods used in the previous sections. First, we need a  lemma, reminiscent of Theorem \ref{Thm1}.

\begin{lem} \label{L4} Let $\mathbb{W}$ be a zero-dimensional Polish space, $S\in\fsigma$, $X,Y$ be recursively presented spaces, and $C\subseteq X$, $U\subseteq Y$ and $B\subseteq X\times Y$, $\Sigma^1_1$ sets. If 
$(\overline{C}^\taudos\times U)\cap B\neq \emptyset$, 
then there are continuous functions $f:\mathbb{W}\to X$ and $g:\mathbb{W}\backslash S \to U$ such that:
\[S\subseteq f^{-1}(C)\]
\[\Delta(\mathbb{W}\backslash S)\subseteq (f\times g)^{-1}(B).\]
\end{lem}

\begin{proof}

Note we can assume $U=Y$, as otherwise, one can take $B\cap (X\times U)$ instead of $B$. So let $N:=(\overline{C}^\taudos\times Y)\cap B$.

Let $\{C_n\}_{n\in\omega}$ be a 1-disjoint family, such that $S=\cup_{n\in\omega} C_n$.

Fix a basis of clopen subsets of $\mathbb{W}$, $\{N_s|s\in 2^{<\omega}\}$ such that $N_s\subseteq N_t$ if $t\sqsubseteq s$, and $N_s\cap N_t=\emptyset$ if $s$ is not compatible with $t$. In particular, for each $x\in\mathbb{W}$, there is an unique $\alpha\in \cantor$ such that $x\in \cap N_{\alpha|k}$.

As in the proof of Theorem \ref{Thm1} define $n_s:=\mbox{min}\{n| N_s\cap C_n\neq \emptyset\}$,  $D:=\{s\in 2^{<\omega}|s=\emptyset\lor n_s\neq n_{s^m}\}$, and $s^-:=s|_{\mbox{max}\{n< |s| | s_{|n}\in D\}}$ if $s\neq\emptyset$, and $s^-:=s$ if $s=\emptyset$.

For each $s\in 2^{<\omega}$, construct the following:
\begin{enumerate}
	\item[i.] $x_s\in X$, and, for each $s\in D$, $y_s\in Y$,
	\item[ii.] $\semirecursive$ sets $X_s\subseteq X$ and, for each $s\in D$, $Y_s\subseteq Y$,
	\item[iii.] for each $s\in D$, a $\Sigma^1_1$ set $V_s\subseteq X\times Y$.
\end{enumerate} 

We ask these sets to satisfy the following conditions:
\begin{enumerate}
	\item\label{cc1} $x_s\in X_s$, $y_t\in Y_t$, and $(x_s,y_s)\in V_s$
	\item\label{cc2} $\mbox{diam}(X_s)<2^{-|s|}$, $\mbox{diam}(Y_s)<2^{-|s|}$, and $\mbox{diam}_{GH}(V_{s})<2^{-|s|}$
	\item\label{cc3} $\left\{\begin{array}{ll}
	\overline{X_{s\varepsilon}}\subseteq X_s & \mbox{if } \varepsilon\in 2,\\
	\overline{Y_s}\subseteq Y_{s^-} &\mbox{if } s\in D\backslash\{\emptyset\}\\
	V_{s}\subseteq V_{s^-} &\mbox{if } s\in D\end{array}\right.$
	\item\label{cc4} $V_{s}\subseteq N\cap\Omega\cap(X_{s}\times Y_{s})$
	\item\label{cc5} if $s\in D$ and $s\varepsilon\notin D$, then $x_{s\varepsilon}\in C\cap \overline{\Pi_0[V_s]}\cap X_s$
	\item\label{cc6} $x_{s}=x_{s^m}$ if  $s\notin D \land s^m\notin D$.
\end{enumerate}

Suppose that these objects have been constructed. If $\alpha\in \mathbb{W}$, define
\[\{f(\alpha)\} :=\bigcap_{\alpha\in N_s} \overline{X_s}.\]

And if $\alpha\notin \cup C_n$, then there is an increasing sequence $(s_k)$ in $D$ such that $\alpha\in\cap_{k\in\omega} N_{s_k}$, so define, by condition \ref{cc3}, \[\{g(\alpha)\}:=\bigcap_{k\in \omega} \overline{Y_{s_k}},\]
\[F(\alpha):=\bigcap_{k\in \omega} V_{s_k}.\]

Note that these functions are continuous, $x_{s_k}\to f(\alpha)$ and $y_{s_k}\to g(\alpha)$ for any strictly increasing sequence $(s_k)$ (with $s_k\in D$ in the second case) such that $\alpha\in\cap_{k\in\omega} N_{s_k}$. 

Note that, by condition \ref{cc6}, if $\alpha\in C_n$, and $s_k$ is an increasing sequence such that $\alpha\in N_{s_k}$, then $f(\alpha)=x_{s_k}\in C$, for $k$ big enough.

If $\alpha\notin \cup C_n$, then $(x_{s_k},y_{s_k})\to F(\alpha)$ in the Gandy-Harrington topology, so in the usual topology as well. In particular $(f(\alpha),g(\alpha))=F(\alpha)\in V_{s_k}\subseteq B$.

To construct the points, we proceed by induction of the length of the sequences, as usual. Note there is $(x_\emptyset,y_\emptyset)\in N\cap \Omega$, since $N$ is a non-empty $\Sigma^1_1$ set. Choose $\semirecursive$ neighborhoods  $X_\emptyset,Y_\emptyset$ of small diameter. We can also choose $V_\emptyset$ a $\Sigma^1_1$ neighborhood of $(x_\emptyset,y_\emptyset)$ of small diameter contained in $N\cap (X_\emptyset\times Y_\emptyset)\cap \Omega$.

Suppose that everything is constructed for $s$ such that $|s|\leq l$. We proceed by cases.

Suppose first that $s\in D$ and $\varepsilon\in 2$. If $s\varepsilon\in D$, define $(x_{s\varepsilon},y_{s\varepsilon}):=(x_s,y_s)$. We only need to shrink the neighborhoods to respective $X_{s\varepsilon},Y_{s\varepsilon}$ and $V_{s\varepsilon}$, and all conditions will be met. Note that $x_s\in  \overline{C}^{\taudos}\cap X_s\cap\overline{\Pi_0[V_s]}$. Since $X_s\cap\overline{\Pi_0[V_s]}$ is $\taudos$-open, there is $x\in C\cap X_s\cap\overline{\Pi_0[V_s]}\cap \Omega$. So, if $s\varepsilon\notin D$, define $x_{s\varepsilon}=x$.  Shrink the neighborhood $X_s$ to one of good diameter, and they will satisfy the relevant conditions.

Suppose now that $s\notin D$. If $s\varepsilon\notin D$, by condition \ref{cc6}, we must define $x_{s\varepsilon}:=x_s$. Again, shrink the neighborhood, and they will satisfy the relevant conditions. Note there is $\varepsilon'\in 2$, such that $s^-\varepsilon'\sqsubseteq s$. It is clear by applying condition \ref{cc6} enough times that $x_{s^-\varepsilon'}=x_s$. In particular, $x_s\in X_s\cap \overline{\Pi_0[V_{s^-}]}$. Thus, there is $x\in X_s\cap \Pi_0[V_{s^-}]$, and $y\in Y_{s^-}$ such that $(x,y)\in V_{s^-}$. If $s\varepsilon\in D$, note $(s\varepsilon)^-=s-$, so define $(x_{s\varepsilon},y_{s\varepsilon}):=(x,y)$. Find neighborhood $X_{s\varepsilon}$, $Y_{s\varepsilon}$, $V_{s\varepsilon}$ such that $V_{s\varepsilon}\subseteq (X_{s\varepsilon}\times Y_{s\varepsilon})\cap V_{s^-}\subseteq (X_{s}\times Y_{s^-})\cap V_{s^-}$. This will clearly satisfy our conditions. 

Thus we can construct all our objects.
\end{proof}

Now, consider a zero-dimensional Polish space $\mathbb{W}$, and let $\{F_{n,\varepsilon}\}$ be a left branching scheme. In each $F_{n,1}$, consider a $1$-disjoint family $\{C^n_m\}_{m\in\omega}$. Fix a dense countable $\{\alpha^{n,m}_k|k\in\omega\}$ subset of $C^n_m$.  Define the following zero-dimensional Polish spaces:

\[\bbX:= \big(\mathbb{W}\backslash(\cap_{n\in\omega}F_{n,0})\big)\oplus\omega,\]
\[\bbY:= \big(\mathbb{W}\backslash (\cup_{n,m\in\omega} C^n_m)\big)\oplus \omega.\]

If $<,,>$ is a bijection from $\omega^3$ onto $\omega$, then define, in $\bbX\times\bbY$,

\[\bbA:=\big\{(\overline{0},\overline{\alpha})|\alpha\in \cap_{n\in\omega} F_{n,0}\big\}\bigcup\big\{(\overline{\alpha^{n,m}_k},\overline{<n,m,k>})|n,m,k\in\omega\big\},\]
\[\bbB:=\big\{(\overline{\alpha},\overline{\alpha})\in\bbX\times\bbY| \alpha\notin (\cap_{n\in\omega} F_{n,0})\cup(\bigcup_{n,m\in\omega } C^n_m)\}.\] 

\begin{Pro} Let $X,Y$ be Polish spaces, and $A,B$ be disjoint analytic subsets of $X\times Y$. At least one of the following must hold:
\begin{enumerate}
	\item $A$ est separable from $B$ by a $\gdelta\times\open$ set,
	\item $(\bbX,\bbY,\bbA,\bbB)\leq (X,Y,A,B)$.
	
	If, in addition, $\{F_{n,e}\}$ converges and $\cup_{m\in\omega}C^n_m$ is dense and meager in $F_{n,1}$, then at most one of the previous can hold.
\end{enumerate}
\end{Pro}

\begin{proof}
Suppose that $A$ is not separable from $B$ by a $\gdelta\times\open$ set.  As usual, we can assume $X,Y$ to be recursively presented, and $A,B$ to be $\Sigma^1_1$. 

Then, by Lemma \ref{L3}, there is $y\in\Pi_1[A]$, such that, for every open neighborhood $V$ of $y$, $\big(\overline{\Pi_0[A]}^{\tau_2}\times V\big)\cap B\neq \emptyset$.

We fix a decreasing basis of neighborhoods $\{V_n\}$ at $y$ and $x\in X$ such that $(x,y)\in A$. Note that, for every $k\in\omega$, $N_n:=\big(\overline{\Pi_0[A]}^{\tau_2}\times V_n\big)\cap B$
is not empty.

By Lemma \ref{L4}, we can find, for each $V_n$, continuous functions $f_n:F_{n,1}\to X$ and $g_n:F_{n,1}\backslash(\cup_{m\in\omega} C^n_m)\to V_n$ such that
\begin{enumerate}
	\item $f_n(\alpha)\in \Pi_0[A]$ if $\alpha\in \cup_{m\in\omega} C^n_m$,
	\item $(f_n(\alpha),g_n(\alpha))\in B$, if $\alpha\in F_{n,1}\backslash (\cup_{m\in \omega} C^n_m)$.
\end{enumerate}
So, for each $\alpha^{n,m}_k$, there is a $y_{n,m,k}$ such that $(f_n(\alpha^{n,m}_k),y_{n,m,k})\in A$. Define $f:\bbX\to X$ and $g:\bbY\to Y$ by:

\[f(\varepsilon,\alpha)=\left\{\begin{array}{cc}
			x & \mbox{if }  \varepsilon=1,\\
			f_n(\alpha) & \mbox{if } \varepsilon=0 \;\&\; \alpha\in F_{n,1}.\\
\end{array}\right.\]

\[g(\varepsilon,\alpha)=\left\{\begin{array}{cc}
			y_{n,m,k} & \mbox{if }  \varepsilon=1 \;\&\; \alpha=<n,m,k>,\\
			g_n(\alpha) & \mbox{if } \varepsilon=0 \;\&\; \alpha\in F_{n,1},\\
			y & \mbox{if } \varepsilon=0 \;\&\; \alpha\in \cap F_{n,0}.
\end{array}\right.\]

The continuity of $f$ is clear, since $\{F_{n,1}\}$ is a partition of $\bbW\backslash(\cap_{n\in\omega} F_{n,0})$ into clopen sets. For $g$, the continuity must only be checked in $\cap F_{n,0}$. If $\alpha_k\to\alpha$, then we can assume that $\alpha_k\in F_{n_k,1}$ for some $n_k$ which must increase as $k$ tends to infinity. In that case, note that $g(\overline{\alpha_k}) \in V_{n_k}$, so that it converges to $y=g(\overline{\alpha})$. 

It is routine to check that $f\times g$ is as a reduction.

To show the second part, we show $\bbA$ is not separable from $\bbB$. So suppose that $\bbA\subseteq U\times V$ with $U\in\gdelta(\bbX)$ and $V\in\open(\bbY)$. We note that $\cap F_{n,0} =\{y\}$, and that $\{F_{n,0}\}$ is a family of clopen neighborhoods of $y$. Therefore, there is $n\in\omega$ such that  $F_{n,0}\subseteq \pi_0[V]$, where $\pi_0$ is the canonical projection from $\mathbb{W}\oplus\omega$ to $\mathbb{W}$. 

Also, note that $\cup C^{n+1}_m$ is a dense set in $F_{n+1,1}\cap \pi_0[U]$. Then $F_{n+1,1}\cap \pi_0[U]$ is a $\gdelta$ dense set in $F_{n+1,1}$. Since $\cup C^{n+1}_m$ is meager in $F_{n+1,1}$, there is $\alpha\in (F_{n+1,1}\backslash\cup_{m\in\omega} C^{n+1}_m)\cap \pi_0[U]$. In particular $(\overline{\alpha},\overline{\alpha})\in (U\times V)\cap B$.
\end{proof}

We are ready to give an example of a basis for this case. Take $\mathbb{W}=\cantor$, $F_{n,\varepsilon}=N_{0^n\varepsilon}$. Let $(s_m)_{m\in\omega}$ be an enumeration of all finite sequences in $2^{<\omega}$ which do not end with $0$. Let $C^n_m=\{0^n1s_m0^\infty\}$, so that $\alpha^{n,m}_k=0^n1s_m0^\infty$. It is easy to see these sets satisfy everything that is needed. In particular, our examples are	
\[\bbA:=\big\{(\overline{0},\overline{0^\infty})\big\}\bigcup\big\{(\overline{0^n1s_m0^\infty},\overline{<n,m,k>})|n,m,k\in\omega\big\},\]
\[\bbB:=\big\{(\overline{\alpha},\overline{\alpha})| \exists^\infty i (\alpha(i)=1)\}.\]

\section{Separation by a $\close\times\fsigma$ set}

Let $\bbW$ be a zero-dimensional Polish space, $\{F_n|n\in\omega\}$ be a $0$-disjoint partition of $\bbW$, $\bbP=\{\alpha_n|n\in\omega\}$ be a countable dense and co-dense subset in $F_0$, and $\{F^{n+1}_{m,\varepsilon}\}$ a left-branching scheme in $F_{n+1}$. Let $\bbX:=\bbW\backslash \bbP$ and $\bbY:=\bbW\backslash (\cup_{n\in\omega}\cap_{m\in\omega}F^{n+1}_{m,0})$. Define the following sets in $\bbX\times\bbY$:
\[\bbA:=\Delta(F_0\backslash \bbP)\cup\Delta\big(\cup_{n,m\in\omega}F^{n+1}_{m,1}\big),\]
\[\bbB:=\{(\alpha,\alpha_n)| (\alpha\in\cap_{m\in\omega} F^{n+1}_{m,0}) \land (n\in\omega)\}.\]

\begin{thm} Let $X,Y$ be Polish spaces and $A,B$ be disjoint analytic subsets of $X\times Y$. At least one of the following holds:
\begin{enumerate}
	\item $A$ is separable from $B$ by a $\close\times\fsigma$ set,
	\item $(\bbX,\bbY,\bbA,\bbB)\leq (X,Y,A,B)$.
	
	If, in addition, all left-branching schemes converge, at most one of the previous holds. 
\end{enumerate}
\end{thm}

\begin{proof} 
As usual, we can suppose $X,Y$ to be recursively presented spaces, and $A,B$ to be $\Sigma^1_1$ sets. Suppose they are not separable. By a Lemma \ref{L3},
 there is a $y\in\Auno$ such that for any $\taudos$-open neighborhood $V$, $(\overline{\Acero}\times V)\cap B\neq \emptyset$. In other words, $\Auno\cap\overline{\Pi_1[(\overline{\Acero}\times Y)\cap B]}^\taudos\neq\emptyset$.

Note also that $\bbP$ is a $\fsigma$ subset of $F_0$, so it is the union of a $1$-disjoint family. Applying lemma \ref{L4}, let $f_0:F_0\backslash\bbP\to X$ and $g_0:F_0\to Y$ be the respective continuous functions.

Note that for each $n$, $g_0(\alpha_n)\in \Pi_1[(\overline{\Acero}\times Y)\cap B]$. Thus, there are $x_n\in X$ and for each $m$, $(x^n_m,y^n_m)\in A$, such that $(x_n,g_0(\alpha_n))\in B$ and $x^n_m\to x_n$ as $m\to\infty$.

With all these fixed, define the functions $f:\bbX\to X$ and $g:\bbY\to Y$ by:

\[
	f(\alpha)=\left\{\begin{array}{lc}
							f_0(\alpha) & \mbox{if } \alpha\in F_0,\\
							x_n & \mbox{if } \alpha\in \cap_{m\in\omega} F^{n+1}_{m,0},\\
							x^n_m & \mbox{if } \alpha\in F^{n+1}_{m,1},
							\end{array}\right.
\]
				
\[	g(\alpha)=\left\{\begin{array}{lc}
							g_0(\alpha) & \mbox{if } \alpha\in F_0,\\
							y^n_m & \mbox{if }\alpha\in F^{n+1}_{m,1}.
							\end{array}\right.
\]

These are obviously continuous and it is easy to see $f\times g$ is a reduction.

Suppose that the schemes converge, and that $\bbA\subseteq C\times V$ with $C$ closed and $V$ a $\fsigma$ set. Note that for any $n$,  if $\{\alpha\}=\cap_{m\in\omega} F^{n+1}_{m,0}$, then $\alpha\in \overline{\Pi_0[A]}$, so $\alpha\in C$. On the other hand, $\Pi_1[A]\cap F_0=F_0\backslash\bbP$ which is not $\fsigma$. In particular, there is $\alpha_n\in \bbP\cap V$. This shows that $(C\times V)\cap \bbB$ is not empty.
\end{proof}

Let $\bbW=\baire$, $F_n=N_n$, $\bbP=\{\alpha\in N_0|\forall^\infty i \alpha(i)=0\}$, $F^{n+1}_{m,0}=N_{(n+1)0^{m+1}}$ and $F^{n+1}_{m,1}=\cup_{k\neq 0}N_{(n+1)0^{m}k}$. In $\bbX\times\bbY$ define 

\[\bbA:=\Delta\big(\{\alpha\in\baire|\big(\alpha(0)=0 \land\exists^\infty i (\alpha(i)\neq 0)\big)\lor\big(\exists n,m,k((n+1)0^m(k+1)\subseteq \alpha)\big)\}\big),\]
and,
\[\bbB:=\big\{\big((n+1)0^\infty,\alpha_n\big)|n\in\omega\big\}.\]

So this provides us a concrete example for a basis in this case.

\nocite{*} 
\bibliographystyle{plain} 
\bibliography{Biblio}

\end{document}